\documentclass[preprint,12pt]{elsarticle}

\usepackage{amssymb}
\usepackage{amsmath}
\usepackage{amsthm}

\journal{SIAM J. of Matrix Analysis and Applications}

\newtheorem{rem}{Remark}
\newtheorem{thm}{Theorem}
\newtheorem{lem}[thm]{Lemma}
\newtheorem{cor}{Corollary}
\newtheorem{Ex}{Example}
\newcommand{\eq} [1] {\begin{equation}\label{#1}}
\newcommand{\en} {\end{equation}}
\newcommand{\KK} {\mathcal{K}}
\newcommand{\Cone}{\operatorname{Cone}}
\newcommand{\trace}{\operatorname{trace}}

\newcommand{\Z} {\mathbb{Z}}
\newcommand{\R} {\mathbb{R}}
\newcommand{\Rn} {\mathbb{R} ^ {n \times n}}
\newcommand{\pRt} { \mathbb{R}^{2 \times 2} _ V}
\newcommand{\pRn} { \mathbb{R}^{n \times n} _ V}
\renewcommand{\Im}{\operatorname{Im}}
\renewcommand{\Re}{\operatorname{Re}}
\newcommand{\diag}{\operatorname{diag}}
\newcommand{\abs}[1]{\left\vert#1\right\vert}
\begin{document}

\begin{frontmatter}



\title{On common invariant cones for families of matrices\tnoteref{support}}
\tnotetext[support]{Research of all authors was supported in part by NSF grant
DMS-0456625. Research of LR was supported in part also by a Summer Research Grant from the College of
William and Mary.}

\author[WM]{Leiba Rodman\corref{cor3}}
\ead{lxrodm@math.wm.edu}
\address[WM]{Department of Mathematics, College of William and Mary,  Williamsburg, VA 23187-8795, USA}

\author[LA]{Hakan Seyalioglu}
\ead{hseyalioglu@ucla.edu}
 \address[LA]{Department of Mathematics, UCLA, Los Angeles, CA 90024, USA}

\author[WM]{Ilya M.  Spitkovsky}
\ead{ilya@math.wm.edu, imspitkovsky@gmail.com}
 \cortext[cor3]{Corresponding author}

\begin{abstract}The existence and construction of common invariant cones for families of real matrices is considered. The
complete results are obtained for  $2\times 2$ matrices (with no
additional restrictions) and for families of simultaneously
diagonalizable matrices of any size. Families of matrices with a shared
dominant eigenvector are considered under some additional
conditions.
\end{abstract}


\end{frontmatter}


\section{Introduction}\label{Intro}

The theory of nonnegative matrices, and more generally of matrices
that leave invariant a convex, closed, pointed, solid cone, is classical;
we mention here the books \cite{BePle94, BaRa97} among many
others; see also \cite{Tam01} for a review of many results, including
recent ones, and extensive bibliography. More generally, real matrices
that leave invariant a convex, closed, pointed, solid cone, have been
studied in \cite{Bir67, Vander, TaSch, Tam04, Bar72, VaFa}. A complete
characterization of such matrices in terms of spectral structure was
obtained in \cite{Vander}. An interesting  application to the multiple
agents randezvous problem is given in \cite{TiFu}.

Recently, several works appeared studying  matrices having common
invariant convex, closed, pointed, solid cones. These works have been
motivated primarily by applications in Glass networks \cite{EdMcDT}
and joint spectral radius \cite[Theorem 1]{BloNe05}. Glass networks are continuous-time
switching networks used to model gene regulatory networks and
neural networks; see \cite{EdMcDT} and references there for an in
depth discussion on Glass networks.

The paper \cite{EdMcDT} actually served as a motivation for the
current paper. We develop here results on matrices having common
invariant cones. The auxiliary Section~\ref{prelim} contains necessary
notions and definitions, in particular that of a proper cone and a
dominant eigenvector. In Section~\ref{com2by2},  a full description is
given of families of $2 \times 2$ real matrices having common
invariant proper cones. As it turns out even in this case the
characterizations are rather involved, and the proofs not immediate.
Some partial results (for pairs of diagonalizable but not
simultaneously reducible matrices) in this venue were obtained in
\cite{EdMcDT}. Our approach is based on the description of all
invariant cones for a single $2\times 2$ matrix given in
Section~\ref{2by2}. In spite of its elementary nature, we did not find
this description in the literature, and include it for the sake of self
containment. Section~\ref{diag} contains the existence criterion for
(and actually a construction of)  a common invariant cone of a family
of simultaneously diagonalizable matrices, while
Section~\ref{common} provides some sufficient conditions for such a
cone to exist when the matrices share the dominant eigenvector.
Finally, Section~\ref{examples} consists of several examples
illustrating both the results obtained and their limitations.

\section{Preliminaries and definitions}\label{prelim}

Let $\mathbb{R}$ be the field of real numbers, $\mathbb{R}^n$ the set
of real $n$-component column vectors, and $\mathbb{R}^{m \times
n}$ the set of real $m\times n$ matrices. {\em All matrices in the
present paper are assumed to be real, unless explicitly stated
otherwise.} A set $\KK \subseteq \mathbb{R}^n$ is a {\em cone} if $a
\KK \subseteq \KK$ for all scalar multiples $a \geq 0$. A cone $\KK$ is
said to be {\em proper} if $\KK+\KK\subseteq\KK$ (so that  $\KK$ is
{\em convex}), {\em closed}, {\em pointed} ($\KK \cap -\KK = \{0\}$)
and {\em solid} (the interior of $\KK$ is nonempty). 

For  $X$ being a subset of $\R^n$ or $\mathbb{R}^{m \times n}$, we
denote by $\Cone X$ the smallest convex cone containing $X$ and
say that $X$ {\em generates} $\Cone X$. Of course, $\Cone X$ is
nothing but the set of all (finite) linear combinations of elements of
$X$ with non-negative coefficients. A cone having a finite generating
set is called {\em polyhedral}. Polyhedral cones are always closed.

For a square matrix $A$,  by the {\em degree} of its eigenvalue
$\lambda$ in this paper we understand  its multiplicity as a root of the
minimal polynomial of $A$ (that is, the size of the largest block, in the
Jordan canonical form of the matrix, corresponding to the eigenvalue
$\lambda$).  We will denote the eigenvalues of an $n\times n$
matrix $A$ by $\lambda_1(A),\ldots,\lambda_n(A)$ (or simply by
$\lambda_1,\ldots,\lambda_n$ if the choice of the matrix is clear from
the context), always taking $\rho(A)=\lambda_1$ provided that  the
{\em spectral radius} $\rho (A)$ of $A$ is an eigenvalue. We will call
the respective eigenvector (eigenspace) the  {\em dominant
eigenvector} (resp., {\em dominant eigenspace}) of $A$.  In case when
an eigenspace is one dimensional, we will (naturally) call it an {\em
eigenline}. We will also use the term {\em eigenray} for each of the
two rays into which an eigenline is partitioned by the origin. Finally
$\sigma (A)$ will be used to denote the set of all eigenvalues of $A$.

A cone $\KK\subseteq\mathbb{R}^n$ is said to be {\em invariant}
under $A\in \Rn$ if $Ax\in \KK$ for every $x\in \KK$. The following
remark is trivial, but will be useful in our analysis.

\begin{rem} \label{nov121} A cone $\KK={\rm Cone}\, \{v_1,\ldots, v_m\}$ is
$A$-invariant if and only if $Av_j\in \KK$ for $j=1,2,\ldots, m$.
\end{rem}

The following result was proved by Vandergraft \cite{Vander}.
\begin{thm}\label{th:V}
$A \in \Rn$ has an invariant proper cone if and only if
\begin{itemize}
\item[\em{(i)}] The spectral radius $\rho (A) \in \sigma(A)$, and
\item[\em{(ii)}] $\deg\lambda_1(A) \geq \deg \lambda_i(A)$ for
    every eigenvalue $\lambda_i(A)$  with $|\lambda_i(A)| =
    \lambda_1(A)$.
\end{itemize} If conditions {\em (i)-(ii)} hold, then also \begin{itemize}\item[\em (iii)] Any $A$-invariant
proper cone contains a dominant eigenvector of $A$.\end{itemize}
\end{thm}

For spectral criteria for existence of polyhedral proper invariant cones see 
\cite{TaSch, VaFa}.
 
We will be using the term {\em Vandergraft matrices} for  real matrices
satisfying conditions (i) and (ii) of Theorem \ref{th:V}, denoting the set
of all such $n\times n$ matrices by $\pRn$.

\section{Invariant proper cones for \boldmath{$2 \times 2$} matrices}\label{2by2}

It is very easy to characterize matrices in $\pRt$. Namely, condition (i)
of Theorem~\ref{th:V} is  equivalent to \eq {v2} (\trace A)^2\geq 4\det
A, \quad \trace A\geq 0, \en the first inequality in (\ref{v2}) meaning
simply that the eigenvalues of $A$ are real while the second inequality
guarantees that the one with the bigger absolute value is
non-negative. Since condition (ii) then holds automatically, a $2\times
2$ matrix $A$ is Vandergraft if and only if it satisfies (\ref{v2}).

Conditions (\ref{v2}) hold, in particular, when both eigenvalues $\lambda_1,\lambda_2$ of $A$ are
non-negative. Description of all $A$-invariant proper cones in this
case is given by the following two theorems, dealing with
diagonalizable and non-diagonalizable matrices $A$ separately.  Of
course, in the former situation only the case $\lambda_1\neq
\lambda_2$ is of interest, because otherwise $A$ is a scalar matrix
which leaves every cone invariant.

\begin{thm}\label{th:pd}Let a $2\times 2$ matrix $A$ be diagonalizable, with $\lambda_1>\lambda_2\geq 0$.
Then a proper cone $\KK\subset\R^2$ is $A$-invariant if and only if it
contains an eigenvector of $A$ corresponding to $\lambda_1$ and its
interior does not intersect the eigenline of $A$ corresponding to
$\lambda_2$. \end{thm}
\begin{proof}{\sl ``Only if" part}. An $A$-invariant
proper cone $\KK$ must contain an eigenvector of $A$
corresponding to $\lambda_1$, as follows from Theorem~\ref{th:V},
part (iii).  Denote this vector by $u_1$ and suppose for a moment that
there is an eigenvector $u_2$ of $A$ corresponding to the eigenvalue
$\lambda_2$ and lying in the interior of $\KK$. Then for sufficiently
large $M>0$ also $-u_1+Mu_2\in \KK$, and for all $n=1,2,\ldots$,
\[  (\lambda_1^{-1}A)^n(-u_1+Mu_2)=-u_1+M(\lambda_2/\lambda_1)^nu_2\in\KK. \]
Letting $n\to\infty$, from the closedness of $\KK$ we conclude that
$-u_1\in\KK$. This, however, contradicts pointedness of $\KK$.

{\sl ``If" part}. Any proper cone in $\R^2$ is  generated by two
linearly independent vectors: $A=\Cone\{v_1,v_2\}$. The conditions
imposed on $\KK$ mean that, after appropriate scalings, its
generating vectors can be written as
\[ v_1=u_1+u_2,\quad v_2=u_1-xu_2,\] where $x\geq 0$.
(Here $u_1$, $u_2$ are eigenvectors corresponding to $\lambda_1$, $\lambda_2$, respectively.) Then \[ Av_1
(=\lambda_1u_1+\lambda_2u_2) =
\frac{x\lambda_1+\lambda_2}{1+x}v_1+\frac{\lambda_1-\lambda_2}{1+x}v_2\in\KK
\] and \[ Av_2 (=\lambda_1u_1-x\lambda_2u_2) =
\frac{x(\lambda_1-\lambda_2)}{1+x}v_1+\frac{\lambda_1+x\lambda_2}{1+x}v_2\in\KK.
\] $A$-invariance of $\KK$ therefore  follows from Remark~\ref{nov121}.
\end{proof}

Let now $A\in\R^{2\times 2}$ be non-diagonalizable. Then, for any
$v\in\R^2$, \eq{av} Av=\lambda v+xu,\en where  $\lambda$ is the
eigenvalue of $A$, $u$ is its (arbitrarily fixed) eigenvector, and $x
\in\R$. We will say that $v$ is {\em positively/negatively associated}
with $u$  ({\em relative to} $A$, if there is a need to mention the matrix
explicitly) if in (\ref{av}) $\pm x>0$. Observe that $x=0$ if and only if
$v$ belongs to the eigenline of $A$, that is, is a scalar multiple of $u$.

Of course, $v$ is positively associated with $u$ if and only if $-v$ is
negatively associated with $u$ if and only if $-v$ is positively
associated with $-u$. Geometrically speaking, the plane $\R^2$ is
partitioned by the eigenline of $A$ into two open half-planes; one
consisting of vectors positively associated with $u$, and the other of
vectors negatively associated with $u$.

\begin{thm}\label{th:pnd}Let $A\in\R^{2\times 2}$ be a non-diagonalizable matrix
with the eigenvalue $\lambda\geq 0$. Then a proper cone $\KK$ is
$A$-invariant if and only if it is given by $\KK=\Cone\{u,v\}$, where $u$
is an eigenvector of $A$ and $v$ is positively associated with $u$
relative to $A$.
\end{thm}
\begin{proof}{\sl ``If" part}. Since $\lambda\geq 0$, from
(\ref{av}) it follows that $Av\in\Cone\{u,v\}$, because $x\geq 0$.
Obviously, $Au=\lambda u$ also lies in $\Cone\{u,v\}$. The desired
result now follows from Remark~\ref{nov121}.

 {\sl ``Only if" part}. Let a proper cone $\KK$ be $A$-invariant. Due to Theorem~\ref{th:V}(iii),
 there is an eigenvector of $A$ lying in $\KK$. Denoting it by $u$, observe that
vectors  negatively associated with $u$ cannot lie in $\KK$. Indeed, if
$\lambda=0$ and  (\ref{av}) holds with $x<0$, then \[ v\in\KK
\Longrightarrow -u\in\KK,\] which contradicts the pointedness of
$\KK$.
For $\lambda>0$, (\ref{av}) implies \[ A^nv=\lambda^n
v+nx\lambda^{n-1}u,\quad n=1,2,\ldots.\] Consequently, if $v\in\KK$ and
$x<0$, then \[
-u=\lim_{n\to\infty}\frac{1}{n\abs{x}}\lambda^{1-n}A^nv\in\KK
\] --- once again, a contradiction with the pointedness of $\KK$.

Since in every neighborhood of $u$ there are vectors negatively
associated with it, $u$ cannot lie in the interior of $\KK$. Thus, it must
be one of its generating vectors. The other generating vector $v$,
being linearly independent with $u$,  must be positively associated
with it. So, $\KK$ indeed is of the desired form. \end{proof}
\begin{cor}\label{cor:pnd} For non-diagonalizable Vandergraft $2\times 2$ matrices,
the dominant eigenvector lies on the boundary of their invariant
proper cones.\end{cor} As follows from Theorem~\ref{th:pd}, for
diagonalizable $2\times 2$ matrices with positive eigenvalues the
dominant vector can lie both in the interior and on the boundary of
their invariant cones.

We turn now to the remaining case of matrices $A$ with negative
determinants. Denote the eigenvalues of $A$ by $\lambda_1 (>0)$
and $\lambda_2 (<0)$, and let $u_1, u_2$ stand for the respective
eigenvectors.
\begin{thm}\label{th:ne}Let $A\in\R^{2\times 2}$ and $\det A<0$.
Then a proper cone $\KK\subset\R^2$ is $A$-invariant if and only if it
can be represented as $\KK =\Cone\{v_1,v_2\}$, where
\eq{vu}v_j=u_1+c_ju_2 \quad (j=1,2)\en and  \eq{ne} c_1>0,\
c_2<0,\quad
\frac{\lambda_1}{\lambda_2}\leq\frac{c_1}{c_2}\leq\frac{\lambda_2}{\lambda_1}.\en
\end{thm}
\begin{proof}An $A$-invariant pointed cone cannot  contain eigenvectors of $A$
corresponding to a negative eigenvalue. Thus, all vectors $v\in\KK$
(in particular, its generators) in their expansion along the eigenbasis
$\{u_1,u_2\}$ have the same sign coefficients corresponding to
$u_1$. Switching from $u_1$ to $-u_1$ if needed, we may without loss
of generality suppose that these coefficients are positive. Scaling
$v_1$ and $v_2$ if necessary, we arrive at (\ref{vu}). Yet another
change (from $u_2$ to $-u_2$, or flipping $v_1$ with $v_2$) allows us
without loss of generality suppose that $c_1>c_2$.

On the other hand, for $v_j$ given by (\ref{vu}) we have \[
Av_j=\lambda_1(u_1+\frac{\lambda_2}{\lambda_1}c_ju_2).\]
Consequently, $Av_j$ lie in the cone $\KK$ if and only if the numbers
$\lambda_2\lambda_1^{-1}c_j$ lie in $[c_2,c_1]$. This is equivalent
to (\ref{ne}).  \end{proof}

\begin{cor}\label{cor:neg} Let $A$ be a $2\times 2$ Vandergraft matrix with
negative determinant. Then  any dominant eigenvector of $A$ lies in the interior of
those $A$-invariant proper cones that contain the eigenvector.\end{cor}

Note that conditions (\ref{ne}) are consistent if and only if $\det A<0$
and $\trace A\geq 0$, which of course agrees with (\ref{v2}). If this is
indeed the case, for every non-zero vector $v$ different from the
eigenvectors of $A$ there exist $A$-invariant proper cones $\KK$ with
$v$ being one of the generators.  The second generators of these
cones form yet another convex cone, described by (\ref{ne}) with one
of $c_j$ being determined by $v$ and the other serving as a
parameter. The latter cone degenerates into a single ray if and only if
$\trace A=0$ (equivalently: $A^2$ is a scalar multiple of the identity),
when necessarily $c_1=-c_2$.

It is very easy to produce directly an $A$-invariant cone with arbitrarily
chosen generator $v$ for any $2\times 2$ matrices $A$ with \eq{dt}
\det A\leq 0,  \quad \trace A\geq 0.\en
\begin{lem}\label{l:vAv}Let $A\in\R^{2\times 2}$ satisfy {\em (\ref{dt})}. Then
$\KK_v:=\Cone\{v,Av\}$ is $A$-invariant for any $v\in\R^2$, $v\neq 0$.\end{lem}
Of course, $\KK_v$ is proper if and only if $v$ is not an
eigenvector of $A$.
\begin{proof}Indeed, $\KK_v$ is generated by $v$ and $Av$. The first of these generators
is mapped by $A$ into $\KK_v$ by construction, and \[
A(Av)=A^2v=(\trace A)(Av)+(-\det A)v\in\KK_v\] due to the
Cayley-Hamilton theorem. \end{proof} This simple observation will
become useful in the next section.

\section{Common invariant cones for families of \boldmath{$2\times 2$} matrices}\label{com2by2}

Let ${\mathcal A}=\{A_1,\ldots,A_n\}$ be a finite family of $2\times 2$
real matrices. An $\mathcal A$-invariant proper cone by definition is
$A_j$-invariant for all $j=1,\ldots,n$, and in order for that to be possible
each of the $A_j$'s  must be a Vandergraft matrix.

In particular, matrices of the form $cI$ with $c<0$ preclude the
existence of $\mathcal A$-invariant proper cones. On the other hand,
presence (or absence) of matrices $cI$ with $c\geq 0$ in $\mathcal
A$ is irrelevant. All such matrices (if any) can be deleted from
$\mathcal A$ but may as well be left intact.

We first consider the case when all the matrices $A_j$ share a
dominant eigenvector $u$.  If several of them are non-diagonalizable, we will say that they
have the {\em same orientation} if the sets of vectors positively
associated with $u$ relative to these matrices coincide (of course, the
sets of vectors negatively associated with $u$ then coincide as well).
This happens if and only if in a basis containing $u$ the off diagonal
elements of these matrices are all of the same sign.

\begin{thm}\label{th:comeg}Let ${\mathcal A}=\{A_1,\ldots, A_n\}$ be a family of
 $2\times 2$ Vandergraft matrices sharing the same dominant eigenvector $u$.
Then  there exists an $\mathcal A$-invariant proper cone $\KK$   if and
only if either
\begin{itemize}\item[{\em (i)}] all $A_j$ are diagonalizable, and those of them (if
    any) which have \newline $\det A_j<0$, $\trace A_j=0$ are scalar
    multiples of each other, or \item[{\em (ii)}] all $A_j$ have
    non-negative determinants, and those of them which are not
    diagonalizable (if any) have the same orientation.
    \end{itemize}\end{thm}
\begin{proof} {\sl ``If" part.} (i) If all $A_j$ are diagonalizable and have non-negative
determinants, the result follows
from Theorem~\ref{th:pd}: any proper cone $\KK$ containing $u$ and
sufficiently narrow to avoid all the eigenvectors of $A_j$
corresponding to their second eigenvalue will do the job.

Suppose now that some of $A_j$ have negative determinants; relabel
them by $A_1,\ldots,A_k$. Consider
$\KK=\Cone\{u,v,A_1v,\ldots,A_kv\}$, where  $v$ is a vector different
from $u$ but so close to it that $\KK$ does not contain the
non-dominant eigenvectors of $A_{k+1},\ldots,A_n$ and $A_iA_j$
($i,j=1,\ldots,k$). The latter products are all Vandergraft matrices with
positive determinants and, under conditions imposed, also
diagonalizable (here, the hypothesis that all matrices $A_j$ with negative determinants and zero traces
are multiples of each other, is crucial). Hence,
$\KK$ is invariant under $A_{k+1},\ldots,A_n$
and $A_iA_j$ ($i,j=1,\ldots,k$) as in the previous part of the proof (see Theorem \ref{th:pd}). In
particular, $A_iA_1v,\ldots,A_iA_kv\in \KK$ for all $i=1,\ldots,k$. Since
$A_iv\in\KK$ ($i=1,2,\ldots ,k$) by the construction of $\KK$, in fact all the generators
of
$\KK$ are mapped by $A_1,\ldots,A_k$ into $\KK$, so that $\KK$ is
invariant also under $A_i$ for $i=1,\ldots,k$.

(ii) There is no need to consider the case when all $A_j$ are
diagonalizable, because it is covered by (i). Supposing that
non-diagonalizable matrices are present in $\mathcal A$, relabel them
by $A_1,\ldots,A_k$. Choose a vector $v$ positively associated with
$u$ relative to $A_1$; under the conditions imposed it will be
positively associated with $u$ also relative to $A_2,\ldots,A_k$. By
Theorem~\ref{th:pnd}, $\KK=\Cone\{u,v\}$ is $A_j$-invariant for
$j=1,\ldots,k$. Moving $v$ sufficiently close to $u$ in order to avoid
the non-dominant eigenvectors of $A_{k+1},\ldots,A_n$, we will make
$\KK$ invariant with respect to all $A_1,\ldots,A_n$.

{\sl ``Only if" part}. In cases different from (i)--(ii) the family $\mathcal A$
contains either\newline (iii) two linearly independent matrices with
negative determinants and zero traces, or (iv) two non-diagonalizable
matrices with different orientation, or (v) a non-diagonalizable matrix
and a matrix with negative determinant.

Denote the matrices involved in each case by $A_1$ and $A_2$. Then
in case (iii) $A_1A_2$ is a non-diagonalizable Vandergraft matrix, so
that (iii) reduces to (v).  In case (iv), due to the description given by
Theorem~\ref{th:pnd} the intersection of any $A_1$-invariant proper
cone with an $A_2$-invariant proper cone is a ray spanned by $u$,
and therefore not proper. In case (v), the non-existence of common
invariant proper cones follows from the comparison of
Corollaries~\ref{cor:pnd} and \ref{cor:neg}. \end{proof}
\begin{cor}\label{co:two}In the setting of Theorem~{\em \ref{th:comeg}}, an
$\mathcal A$-invariant proper cone exists if and only if any two
matrices in the family $\mathcal A$ share an invariant proper
cone.\end{cor}
\begin{proof}Indeed, from the consideration of cases (iii)--(v)  in the proof of Theorem~\ref{th:comeg}
it follows that there exists  a pair of matrices in  $\mathcal A$ with no
common invariant proper cone, whenever conditions (i) or (ii) do not
hold. \end{proof}

We now move to the situation when $\mathcal A$ contains matrices
with different dominant eigenlines. As it happens, the crucial role is
then played by an extended family ${\mathcal A}_1$ which contains
$\mathcal A$ and all pairwise products (different from scalar multiples
of the identity) of the matrices in $\mathcal A$ having negative
determinants:
\[ {\mathcal A}_1= {\mathcal A}\cup\{ A_iA_j\colon A_i,A_j\in{\mathcal
A},\ \det A_i,\det A_j<0\text{ and } A_iA_j\neq cI \}.\] Of course, ${\mathcal A}_1$ coincides with
$\mathcal A$ if the latter consists only of matrices with non-negative
determinants.
\begin{thm}\label{th:necond}Let $\mathcal A$ be a finite family in $\R^{2\times 2}$.
For  an $\mathcal A$-invariant proper cone to exist it is necessary that
\begin{itemize}\item[{\em (i)}] all elements of ${\mathcal A}_1$ are Vandergraft matrices,
\item[{\em (ii)}]there are at most two dominant eigenlines
    corresponding to \newline  non-diagonalizable matrices in ${\mathcal
    A}_1$, and all of them (if there is more than one) corresponding to
    the  same dominant eigenline also have the same orientation, \end{itemize}
    and
\begin{itemize}\item[{\em (iii)}]  the dominant eigenlines of matrices in ${\mathcal A}_1$
are separated from the non-dominant ones.\end{itemize}
    \end{thm}
The {\em separation} condition  (iii)  means simply the existence of
vectors $v_1,v_2$ such that the interior of $\Cone\{v_1,v_2\}$ is free of
the non-dominant eigenvectors of matrices in ${\mathcal A}_1$  while
the interior of $\Cone\{v_1,-v_2\}$ is free of the dominant eigenvectors
of non-scalar matrices. The vectors $v_j$ themselves are allowed to
be both dominant and non-dominant, but only if as the latter they
correspond to matrices in ${\mathcal A}_1$  with non-negative
determinants. 

\begin{proof}An $\mathcal A$-invariant cone $\KK$ also is ${\mathcal A}_1$-invariant. This
immediately implies the necessity of condition (i).

According to Corollary~\ref{cor:pnd}, an eigenline of a
non-diagonalizable $2\times 2$ Vandergraft matrix must be passing
through the boundary of any of its invariant proper cones. Thus, at
most two such eigenlines are admissible.

If two non-diagonalizable matrices share the eigenline but have
different orientation, the intersection of (any pair of) the respective
invariant cones is an eigenray, due to Theorem~\ref{th:pnd}, and
therefore is not proper. These two observations settle the necessity of
part (ii).

Finally, if $\KK$ is an $\mathcal A$- (and therefore ${\mathcal
A}_1$)-invariant proper cone, then all dominant eigenlines lie in
$\KK\cup(-\KK)$ while non-dominant eigenlines belong to the closure
of the complement (to the complement itself, if the respective matrix
has negative determinant). Thus, (iii) holds.
\end{proof}

Suppose now that necessary conditions stated in
Theorem~\ref{th:necond} hold. Denote by $U=\{u_1,\ldots,u_N\}$ the
set of all distinct dominant unit eigenvectors of matrices in ${\mathcal
A}_1$ the directions of which are chosen in such a way that $\Cone
U$ is proper and its interior is free of non-dominant eigenlines (this is
possible due to (iii)). If there are no such eigenlines (that is, all
matrices in ${\mathcal A}$ are non-diagonalizable), impose instead
the condition that $u_j$ for $j=2,\ldots,N$ are positively associated
with $u_1$ relative to the matrix $A_1$ for which $u_1$ is an
eigenvector (this is possible due to (ii)).  This choice is unique up to
changing the sign of all $u_j$ simultaneously. Relabel also the
elements of $\mathcal A$ in such a way that $\det A_i$ is negative for
$i=1,\ldots, k$ and non-negative otherwise (with the convention that
$k=0$  if $\det A_i\geq 0$ for all $i=1,\ldots,n$).

For further consideration it is convenient to distinguish between the
cases when there is none, one, or two dominant eigenlines
corresponding to non-diagonalizable matrices in ${\mathcal A}_1$.

\begin{thm}\label{th:none}Let ${\mathcal A}=\{A_1,\ldots,A_n\}\subset \R^{2\times 2}$
be such that all the elements of ${\mathcal A}_1$ are diagonalizable
matrices. Under the necessary conditions\footnote{Condition (ii)
holds automatically.} {\em (i), (iii)} of Theorem $\ref{th:necond}$ and
using the notation introduced above, let \eq{bigcone} \KK=\Cone\{u_j,
A_iu_j\colon i=1,\ldots, k;\ j=1,\ldots,N\}.\en  Then there exist
$\mathcal A$-invariant proper cones if and only if the cone $\KK$ is
proper,  its
interior is free of the non-dominant eigenvectors of all matrices in
${\mathcal A}_1$, and the edges of $\KK$ are not collinear with the
eigenvectors of $A_i$ ($i=1,\ldots,k$).
\end{thm}
\begin{proof}{\sl ``Only if" part}. Any $\mathcal A$-invariant cone also is
${\mathcal A}_1$-invariant, and thus must contain either $U$ or $-U$.
Without loss of generality, let it contain $U$. Then, being invariant
under all $A_i$, it must also contain $\KK$. The rest follows from
Theorems~\ref{th:pd} and \ref{th:ne}, applied to each of the matrices
in ${\mathcal A}_1$.

{\sl ``If" part}. For $i=k+1,\ldots,n$, the cone $\KK$ contains the
dominant eigenvector of $A_i$ (since it is one of the $u_j$'s) and the interior of
$\KK$ does not contain its non-dominant eigenvectors. By
Theorem~\ref{th:pd}, $\KK$ is invariant under $A_i$.

Since $\det A_i A_m>0$ for all $i,m=1,\ldots,k$, the cone $\KK$ for the
same reasons is $A_iA_m$-invariant. Consequently,
$A_iA_mu_j\in\KK$ for all $i,m=1,\ldots,k$; $j=1,\ldots,N$. But $A_iu_j$
lies in $\KK$ by construction. So, all the generators of $\KK$ are
mapped into $\KK$ by $A_1,\ldots,A_n$. It remains to invoke
Remark~\ref{nov121}. \end{proof}

Theorem~\ref{th:none} shows that necessary conditions stated in
Theorem~\ref{th:necond} in general are not sufficient. For $k=0$,
however, $\KK$  coincides with $\Cone U$, and the latter is $\mathcal
A$-invariant already under the conditions of Theorem~\ref{th:necond}.
The situation therefore simplifies as follows.
\begin{cor}\label{co:4diagnonneg}Let ${\mathcal A}=\{A_1,\ldots,A_n\}$ be a family of diagonalizable
$2\times 2$  matrices with non-negative determinants. Then in order
for an $\mathcal A$-invariant proper cone to exist it is necessary and
sufficient that \begin{itemize}\item[{\em (i)}] all elements of $\mathcal
A$ are Vandergraft matrices, and \item[{\em (ii)}] the dominant
eigenlines of matrices in $\mathcal A$ are separated from the
non-dominant ones.\end{itemize}
\end{cor}
We can now observe the following. 

\begin{thm}\label{co:4}Let ${\mathcal A}=\{A_1,\ldots,A_n\}$ be a family of diagonalizable
$2\times 2$ Vandergraft matrices with non-negative determinants. If
any four of them (three --- if there is at most one pair of simultaneously
diagonalizable matrices in $\mathcal A$) have a common invariant
proper cone, then there also exists an $\mathcal A$-invariant proper
cone. \end{thm}
\begin{proof}Indeed, if an $\mathcal A$-invariant proper cone does not exist, then condition (iii)
of Theorem~\ref{th:necond} fails. But then it is possible to find four
matrices in $\mathcal A$ (without loss of generality relabel them by
$A_1,\ldots,A_4$) such that, when traveling around the origin in a
counterclockwise direction, one encounters consequently the
dominant eigenline of $A_1$, the non-dominant eigenline of $A_2$,
the dominant eigenline of $A_3$, and finally the non-dominant
eigenline of $A_4$. Condition (iii) fails for the set
$\{A_1,A_2,A_3,A_4\}$, so that these four matrices already do not have
a common invariant proper cone. Of course, it is not excluded that
$A_1$ or $A_3$ coincides with $A_2$ or $A_4$, and then we have an
even smaller subfamily of $\mathcal A$ with no common invariant
proper cone. If $A_1$ and $A_3$ are not simultaneously
diagonalizable, the non-dominant eigenline of at least one of them will
be different from the dominant eigenline of the other. Consequently, in
this case we can always choose $A_2$ or $A_4$ coinciding with
$A_1$ or $A_3$.  A similar reasoning works if a pair $A_2, A_4$ is not
simultaneously diagonalizable.\end{proof} 

We now move to the case of one dominant eigenline corresponding to
non-diagonalizable matrices. 

\begin{thm}\label{th:one}Let ${\mathcal A}=\{A_1,\ldots,A_n\}\subset \R^{2\times 2}$
satisfy conditions {\em (i)--(iii)}  of Theorem $\ref{th:necond}$, with exactly one
dominant direction (say, corresponding to $u_1$) shared by all
non-diagonalizable matrices in ${\mathcal A}_1$. Then there exists an
$\mathcal A$-invariant proper cone if and only if the cone {\em
(\ref{bigcone})} satisfies conditions of Theorem $\ref{th:none}$, and in
addition its interior consists of vectors positively associated with
$u_1$.
\end{thm}
\begin{proof}{\sl ``Only if" part}. An $\mathcal A$-invariant cone is invariant under
all diagonalizable matrices in ${\mathcal A}_1$. Therefore, it must
contain the cone (\ref{bigcone}). On the other hand, it is also invariant
under all non-diagonalizable matrices in ${\mathcal A}_1$, so that by
Theorem~\ref{th:pnd} $u_1$ must lie on its boundary, and the interior
of the cone consists only of vectors positively associated with $u_1$.
The same is therefore true for $\KK$.

{\sl ``If" part}. As in Theorem~\ref{th:none}, the cone (\ref{bigcone})
itself does the job. \end{proof}

The case of two dominant eigenlines corresponding to
non-diagonalizable matrices in ${\mathcal A}_1$ can be treated along
the same lines. However, a more straightforward (and less
computationally consuming) approach also is available.

Suppose that conditions (i), (ii), and (iii) of Theorem~\ref{th:necond}
hold, and that ${\mathcal A}_1$ contains  two non-diagonalizable
matrices (say, $B_1$ and $B_2$) with non-collinear dominant
eigenvectors. Relabel the latter as $u_1$ and $u_2$, choosing the
direction of $u_1$ arbitrarily, and the direction of $u_2$ in such a way
that it is positively associated with $u_1$  relative to $B_1$. According
to Theorem~\ref{th:pnd}, then either $B_1$ and $B_2$ have no
common invariant proper cones (if $u_1$ is negatively associated
with $u_2$ relative to $B_2$), or there are exactly two such cones:
$\KK=\Cone\{u_1,u_2\}$ and $-\KK$.

\begin{thm}\label{th:two}For a finite  family ${\mathcal A}=\{A_1,\ldots,A_n\}$
of Vandergraft matrices with exactly two dominant eigenlines corresponding to
non-diagonalizable matrices in ${\mathcal A}_1$, the only possible
$\mathcal A$-invariant proper cones are $\pm\KK$ introduced above.
These cones are indeed $\mathcal A$-invariant if and only if:
\begin{itemize}\item[{\em (a)}] all non-diagonalizable matrices in ${\mathcal A}$ (if any)
with a dominant eigenvector $u_j$ have the same orientation as
$B_j$ ($j=1,2$), \item[{\em (b)}] for all matrices $A_j\in{\mathcal A}$,
their dominant eigenvectors  lie in $\KK\cup(-\KK)$ while the
non-dominant ones lie outside the interior of $\KK\cup(-\KK)$,
\item[{\em (c)}] for $A=A_j\in{\mathcal A}$ with the eigenvalues
    $\lambda_{1j}>0$, $\lambda_{2j}<0$ and the dominant
    eigenvector $u_1+\xi u_2\in\KK$, the non-dominant eigenvector
    must be collinear with $u_1+\eta u_2$, where \[
    \frac{\lambda_{1j}}{\lambda_{2j}}\leq \frac{\eta}{\xi}\leq
    \frac{\lambda_{2j}}{\lambda_{1j}}. \]
\end{itemize} \end{thm}
\begin{proof}Indeed, conditions (a)--(c) are necessary and sufficient for $\KK$ (or $-\KK$)
to be invariant under all matrices in $\mathcal A$, as follows by
applying Theorem~\ref{th:pd}--\ref{th:ne}. And, as was observed
earlier, no other proper cones can possibly be $\mathcal A$-invariant.
\end{proof}
\begin{rem} It follows directly from the proof of Theorem $\ref{th:two}$ that if in its setting
every three matrices in ${\mathcal A}_1$ (or any five matrices in
$\mathcal A$) have a common invariant proper cone, then there also
exists an $\mathcal A$-invariant proper cone. \end{rem}

\section{Simultaneously diagonalizable matrices}\label{diag}

We now move to square matrices of arbitrary size $m\times m$ but
suppose that all the elements of the family ${\mathcal A}=\{A_1,\ldots,
A_n\}$  under consideration can be put in a diagonal form by the
same similarity transformation $S$ (note that $S$ is allowed to be a
complex matrix). This $S$ then diagonalizes all matrices from
${\mathcal A}_2=\Cone {\mathcal A}$, and moreover from the closed
algebra ${\mathcal A}_3$ generated by ${\mathcal A}$. Denote by $q\
(\leq m)$ the maximal number of distinct eigenvalues for matrices in
${\mathcal A}_2$. If $B_0$ is one of the matrices on which this
number is attained, \eq{b0} B_0=S\diag[b_1 I_{s_1},\ldots,b_q
I_{s_q}]S^{-1},\qquad b_j\in \mathbb{C}, \quad b_i\neq b_j \ \mbox{if}
\ i\neq j,\en then \eq{aj} A_j=S\diag[\lambda_{1j}
I_{s_1},\ldots,\lambda_{qj} I_{s_q}]S^{-1} \text{ for all } A_j\in{\mathcal
A}; \  \text{here} \ \lambda_{ij}\in \mathbb{C}. \en Indeed, if at least
one of the blocks in the middle factor of (\ref{aj})  were different from a
scalar multiple of the identity, then the matrix $B_0+\epsilon A_j$
would have more than $q$ distinct eigenvalues for sufficiently small
$\epsilon$. (Note in passing, though this fact is not needed in what
follows,  that because of (\ref{aj}) $q$ is also the maximal number of
distinct eigenvalues of the matrices in a larger set ${\mathcal A}_3$.)

Assume there exists an  $\mathcal A$-invariant proper
cone. Then obviously all products $A_1^{m_1}\cdots A_n^{m_n}$ ($m_i\in \Z_+$, the set of nonnegative integers) are
Vandergraft matrices. Due to the diagonalizability, this requirement
amounts to
$\max_i\abs{\lambda_{i1}^{m_1}\cdots\lambda_{in}^{m_n}}$ being
attained on some $i$ for which
$\lambda_{i1}^{m_1}\cdots\lambda_{in}^{m_n}\geq 0$.
For every $n$-tuple $(m_1,\ldots, m_n)\in \Z_+^n$,
introduce the set
$$ \Omega(m_1,\ldots, m_n)=\{i_0\in \{1,2,\ldots, q\}\ :\
\max_{1\leq i\leq q} \, |\lambda_{i1}^{m_1}\cdots
\lambda_{in}^{m_n}|= \lambda_{i_01}^{m_1}\cdots
\lambda_{i_0n}^{m_n} \}. $$ Although  $\Omega(m_1,\ldots,
m_n)$ need not be a singleton, we note that there is a unique index
$p=p(m_1,\ldots, m_n)\in \Omega(m_1,\ldots, m_n)$ for which
 $$\max_{i_0\in \Omega(m_1,\ldots, m_n)}\, |b_{i_0}|=b_p. $$ Indeed, this follows from the Vandergraft
property of the matrix
$$ A_1^{m_1}A_2^{m_2}\cdots A_n^{m_n}+B_0$$
and the condition $b_i\neq b_j$ if $i\neq j$.
(Note that we take  $X^0=I$ for every square matrix $X$ regardless if $X$ is singular or not.)
We let $P=\cup p(m_1,\ldots,
m_n)$
where the union is taken over all $n$-tuples $(m_1,\ldots, m_n)\in \Z_+^n$.
Permuting the columns of $S$ if necessary,
we may suppose without loss of generality that this set is
$P=\{1,\ldots,k\}$, where $k\leq q$.

\begin{thm}\label{th:diag}In the notation {\em (\ref{aj})} and for $k$ as introduced
above, $\mathcal A$-invariant proper cones exist if and only if
\eq{feb91} \lambda_{ij}\geq 0  \text{ for all }  i=1,\ldots k \text{ and }
j=1,\ldots,n. \en\end{thm}
\begin{proof}{\sl ``Only if" part}. For an arbitrarily fixed $i_0\in\{1,\ldots,k\}$, pick an
$n$-tuple $m_1,\ldots,m_n$ such that \[
\lambda_{i_01}^{m_1}\cdots\lambda_{i_0n}^{m_n}=
\max_{i=1,\ldots q}
\abs{\lambda_{i1}^{m_1}\cdots\lambda_{in}^{m_n}}.\] Then
\[ \lambda_{i_01}^{m_1}\cdots\lambda_{i_0n}^{m_n}+\epsilon>
\abs{\lambda_{i1}^{m_1}\cdots\lambda_{in}^{m_n}+\epsilon}, \quad
i\not\in \Omega(m_1,\ldots, m_n)\] for any $\epsilon>0$, and therefore
\[ \lambda_{i_01}^{m_1}\cdots\lambda_{i_0n}^{m_n}+\epsilon+\delta b_{i_0}>
\abs{\lambda_{i1}^{m_1}\cdots\lambda_{in}^{m_n}+\epsilon+\delta
b_{i}}, \quad i\neq i_0\] for $\delta>0$ small enough. Having fixed
$\epsilon$ and $\delta\ (>0)$, observe that then for any $j$ such that
$\lambda_{i_0j}\neq 0$,
\[ \abs{(\lambda_{i_01}^{m_1}\cdots\lambda_{i_0n}^{m_n}+\epsilon+\delta b_{i_0})^l \lambda_{i_0j}}>
\abs{(\lambda_{i1}^{m_1}\cdots\lambda_{in}^{m_n}+\epsilon+\delta
b_{i_0})^l\lambda_{ij}}, \quad i\neq i_0\] if the positive integer $l$ is large enough.

In other words,
$(\lambda_{i_01}^{m_1}\cdots\lambda_{i_0n}^{m_n}+\epsilon+\delta
b_{i_0})^l\lambda_{i_0j}$ is strictly bigger (by absolute value) than
other eigenvalues of $$B_l:=(A_1^{m_1}\cdots A_n^{m_n}+\epsilon
I+\delta B_0)^lA_j. $$ But an $\mathcal A$-invariant cone is also
$B_l$-invariant whenever $\epsilon,\delta>0$. So, \[
(\lambda_{i_01}^{m_1}\cdots\lambda_{i_0n}^{m_n}+\epsilon+\delta
b_{i_0})^l\lambda_{i_0j}>0.\] Choosing two consecutive values of $l$,
we conclude that in fact $\lambda_{i_0j}>0$.

{\sl ``If" part}. Denote by $L_+$ the (real) linear span of the first
$s_1+\cdots +s_k$ columns of $S$. Note that since the eigenvalues of
$B_0$ corresponding to these columns of $S$ are real (see (\ref{b0})),
the first $s_1+\cdots +s_k$ columns of $S$ are real as well (or more
precisely can be made real if necessary, by (complex) scalings); thus
$L_+\subset \R^m$. Let us represent $\R^m$ as the direct sum of the
subspaces $L_r$ and $L_c$ spanned respectively by the real columns
of $S$ and by the real and imaginary parts of non-real (if any) columns
of $S$. By definition of $L_+$, it lies in $L_r$. Moreover, $L_r$ can be
written as $L_r=L_+\dotplus L_-$, where $L_-$ is also spanned by
columns of $S$.

Choose bases $F_\pm$ in $L_\pm$ consisting of columns of $S$, and
a basis $F_c$ in $L_c$ consisting of  vectors $u_i,v_i\in\R^m$ such
that \[ A_ju_i=(\Re\lambda_{ij})u_i-(\Im\lambda_{ij})v_i, \quad
A_jv_i=(\Im\lambda_{ij})u_i+(\Re\lambda_{ij})v_i. \] Then of course
\eq{lc}\begin{aligned}(A_1^{m_1}A_2^{m_2}\cdots
A_n^{m_n})u_i & =(\Re\mu_{i})u_i-(\Im\mu_{i})v_i, \\
(A_1^{m_1}A_2^{m_2}\cdots A_n^{m_n})v_i &
=(\Im\mu_{i})u_i+(\Re\mu_{i})v_i, \end{aligned} \en where $m_j\in \Z_+$ and
$\mu_i=\lambda_{i1}^{m_1}\cdots\lambda_{in}^{m_n}$.

Denote by $f$ the sum of all elements in $F_+$, and let $\KK_0$ stand
for the smallest $\mathcal A$-invariant convex cone  containing
$F_+$, $f+F_-$ and $f+F_c$. The span of $\KK_0$ contains the basis
$F=F_+\cup F_-\cup F_c$ of the whole space $\R^m$, so that it
coincides with $\R^m$. In other words, $\KK_0$ is a {\em
reproducing} convex cone, and therefore it is solid.

The closure $\KK$ of $\KK_0$ also is a convex solid cone invariant
under $\mathcal A$. It remains only to show that $\KK$ is pointed.

Let us relabel vectors in $F$ by $f_1,\ldots,f_m$, with the first $p=s_1+\cdots +s_k$
vectors belonging to $F_+$, and denote by $\alpha_j(v)$ the coordinates of
the vector $v$ in its expansion along $F$.

By (\ref{feb91}), for $v=A_1^{m_1}\cdots A_n^{m_n}f_j$, $j=1,\ldots,p$, we have
\[ \alpha_j(v)\geq 0 \text{ and } \alpha_i(v)=0 \text{ for all } i\neq j.\]
Consequently, for such $v$ \eq{env}
\sum_{j=1}^p\alpha_j(v)\geq\sum_{p+1}^m\abs{\alpha_j(v)}.\en
Inequality (\ref{env}) obviously holds for $v\in f+F_-$ or $f+F_c$, since
then the first $p$ coordinates $\alpha_j(v)$ and exactly one of the
other $m-p$ coordinates are equal to one, while the remaining ones
are all zeros. The construction of the subspace $L_+$ (for which $F_+$
is a basis) guarantees that inequality (\ref{env}) persists for vectors
$v$ being images of $f+F_-$ under arbitrary products
$A_1^{m_1}\cdots A_n^{m_n}$. Indeed, the left hand side of (\ref{env})
is \eq{feb92}
\sum_{i=1}^p\lambda_{i1}^{m_1}\cdots\lambda_{im}^{m_n}, \en
while the right hand side is just one summand of the form \eq{feb910}
\abs{\lambda_{j1}^{m_1}\cdots\lambda_{jm}^{m_n}},  \en with $j$
between $p+1$ and $m$. Since all summands in (\ref{feb92}) are
non-negative, and at least one of them is bigger than or equal to
(\ref{feb910}) --- this is where the definition of $L_+$ is being used, ---
inequality (\ref{env}) will hold for such $v$. Moreover, for images of
$f+F_c$ under $A_1^{m_1}\cdots A_n^{m_n}$ we have, due to
(\ref{lc}): \eq{k0}
\sum_{j=1}^p\alpha_j(v)\geq\frac{1}{2}\sum_{p+1}^m\abs{\alpha_j(v)},
\quad \alpha_j(v)\geq 0 \text{ for } j=1,\ldots,p.\en Since inequalities
(\ref{env}) and (\ref{k0}) persist under taking linear combinations with
non-negative coefficients and passing to limits, we see that (\ref{env})
holds in fact for all $v\in\KK$. On the other hand, if (\ref{env}) holds
after switching from $v$ to $-v$, then $\alpha_j(v)=0$ for all
$j=1,\ldots,m$, so that $v=0$. \end{proof}

\section{Families of matrices  with common dominant eigenvector}\label{common}

Theorem~\ref{th:comeg} gives a full treatment of families of $2\times
2$ matrices sharing a dominant eigenvector. In higher dimensions,
however, we have to impose additional restrictions.

\begin{thm} \label{NormCom}
Let $\mathcal A$ be a set of $m\times m$  Vandergraft matrices that
share a common dominant eigenvector $x$ and satisfy at least one of
the following two conditions:
\begin{itemize}
\item[${\rm (1)}$] The matrices in $\mathcal A$ are simultaneously
    similar, with a real similarity matrix, to normal matrices;
\item[${\rm (2)}$] $\mathcal A$ is finite, the matrices in $\mathcal
    A$ commute and for every $A\in {\mathcal A}$, $\rho (A)$ is a
    semisimple eigenvalue, i.e., a simple root of the minimal polynomial, of $A$.
\end{itemize}
Then the matrices in $\mathcal A$ have a common invariant proper
cone $\KK$ with the additional property that $x$ belongs to the
interior of $\KK$.
\end{thm}

For the proof of Theorem \ref{NormCom} we need two lemmas.

\begin{lem}\label{2'}
Let $A_1,\ldots, A_q$ be commuting $m\times m$ real matrices.
Assume that there exists $\lambda_0$ real with the following
properties:
\begin{itemize}
\item[${\rm (1)}$] there exists a nonzero $x$ such that
$A_jx=\lambda_0x$ for $j=1,2,\ldots, q$.
\item[${\rm (2)}$] $\lambda_0$  is a semisimple eigenvalue of
$A_j$, for $j=1,2,\ldots, q$.
\end{itemize}
Then there exists an invertible real matrix $S$ such that $S^{-1} A_jS$ have the form
$$S^{-1}
A_jS=\left[\begin{array}{cc} \lambda_0 &0 \\ 0 & B_j\end{array}\right],
\quad j=1,2,\ldots, q, $$ where $B_1,\ldots, B_q$ are $(m-1)\times (m-1)$
matrices. \end{lem}

\begin{proof} Induction on $q$. For $q=1$, the result is clear. Assume Lemma \ref{2'} has been proved for
$q-1$
matrices. Applying a simultaneous similarity to $A_1, \ldots, A_q$, we may assume that
$$ A_1=\left[\begin{array}{cc} \lambda_0I_p &0 \\ 0 &
\widetilde{A}_1\end{array}\right], $$ where $\lambda_0$ is not an eigenvalue of $\widetilde{A}_1$. Since
$A_1,\ldots, A_q$ commute we have
$$ A_j=\left[\begin{array}{cc} B_j &0 \\ 0 &
C_j\end{array}\right],\quad j=2,3,\ldots, q.$$ Here the matrices $B_2,\ldots, B_q$ are $p \times p$. Clearly,
the vector $x$ (which exists by (1)) has the form $x=\left[\begin{array}{c} y \\ 0
\end{array}\right]$, where $y\neq 0$ has $p$ components.
Then $B_jy=\lambda_0y$. One verifies that $\lambda_0$ is a semisimple eigenvalue of each $B_j$. By the induction
hypothesis, there exists an invertible real $T$ such that
$$ T^{-1}B_jT=\left[\begin{array}{cc} \lambda_0 & 0 \\ 0 &
\widetilde{B}_j \end{array}\right], \quad j=2,3,\ldots, q. $$ Now take $S=\left[\begin{array}{cc} T & 0 \\ 0 & I
\end{array}\right]$ to satisfy the lemma. \end{proof}
\bigskip

\begin{lem}\label{2}
Let $A_1,\ldots, A_q$ be commuting $m\times m$ complex matrices
with the following properties:
\begin{itemize}
\item[${\rm (1)}$] $\rho (A_j)\leq 1$ for $j=1,2,\ldots, q$;
\item[${\rm (2)}$]
every eigenvalue (if exists) on the unit circle of every $A_j$ is semisimple.
\end{itemize}
Then there  exists a positive definite matrix $V$ such that \eq{3} V-A_j^*VA_j\geq 0, \quad {\rm for}\
j=1,2,\ldots, q. \en ($A\geq B$ means that $A-B$ is positive semidefinite).

Moreover, if all $A_j$'s are real, then $V$ can be also chosen real.
\end{lem}

\begin{proof} It is enough to prove the complex case only. Indeed, suppose all $A_j$'s are real and we have
proved that there exists a (generally, complex) positive definite $V$ such that (\ref{3}) holds. Then by taking
complex conjugates in (\ref{3}) we obtain \eq{4} \overline{V}-A_j^T\overline{V}A_j\geq 0, \quad j=1,2,\ldots,
j=1,2,\ldots, q. \en Adding (\ref{3}) and (\ref{4}) we see that $U- A_j^TUA_j\geq 0$, where $U:=V+\overline{V}$
is positive definite and real.

We now prove the complex case. If $\rho (A_j)<1$ for all $j$, let \eq{8} V=\sum (A_1^*)^{z_1}\cdots
(A_q^*)^{z_q}A_q^{z_q}\cdots A_1^{z_1}, \en where the sum is taken over all $q$-tuples $(z_1,\ldots, z_q)$,
$z_j\in \Z_+$. It is easy to see
(using $\rho (A_j)<1$) that the series in (\ref{8})
converges absolutely.  Clearly $V\geq I$ and
\begin{eqnarray*}
V &- &A_j^*VA_j \\ &= & \sum
(A_1^*)^{z_1}\cdots (A_{j-1}^*)^{z_{j-1}}(A_{j+1}^*)^{z_{j+1}}\cdots (A_q^*)^{z_q}A_q^{z_q}\cdots
A_{j+1}^{z_{j+1}}A_{j-1}^{z_{j-1}} \cdots A_1^{z_1}\\[3mm] &\geq & 0, \end{eqnarray*} where the sum is taken over
all $(q-1)$-tuples
$(z_1,\ldots, z_{j-1},z_{j+1},\ldots, z_q)\in \Z_+^{q-1}$.

So suppose that $\rho (A_j)=1$ for some $j$, say $\rho (A_1)=1$. Note that the hypotheses and the conclusions of
Lemma \ref{2} are invariant under simultaneous similarity of $A_1,\ldots, A_q$:
$$ A_j \mapsto  S^{-1}A_jS, \quad j=1,2,\ldots
,q,$$ where $S$ is any invertible $m\times m$ matrix. Then, considering each root subspace of $A_1$ separately,
and taking advantage of the commutativity property $A_jA_k=A_kA_j$ for $j,k=1,2,\ldots, q$, we reduce the proof
to the case $A_1=\lambda I$, $|\lambda|=1$. Then obviously $V-A_1^*VA_1=0$, and it suffices to prove (\ref{3})
for $A_2,\ldots, A_q$. This follows by induction on $q$, the case $q=1$ being easy. \end{proof} \bigskip

We now proceed with the proof of Theorem \ref{NormCom}.

\begin{proof}
Assume first that (1) holds. We may assume that $\mathcal A$
consists of normal matrices and that $\|x\|=1$ (the norm is Euclidean).
Let $\mathbb{M}$ be the orthogonal complement to ${\rm Span}\,
\{x\}$. We claim that:
$$\mathcal{K} := \{c_1 x + y\, : \, c_1\in \mathbb{R}, \ \ y\in \mathbb{M}, \ \  c_1\geq
\|y\|\}$$ is a common invariant proper cone for all $A\in {\mathcal A}$.

Clearly, $\KK$ is a proper cone; therefore  we only have to show that it is invariant with
respect to the matrices. Let $A\in {\mathcal A}$, and let $x,u_2,\ldots,
u_m$ be an orthonormal set with the following properties:
$$ Au_{2k}=\alpha_ku_{2k}+\beta_ku_{2k+1}, \quad Au_{2k+1}=-\beta_ku_{2k} +\alpha_ku_{2k+1}, \quad \mbox{for} \
\ k=1,2,\ldots ,\ell, $$
$$ Au_{s}=\lambda_su_{s} \quad \mbox{for} \ \ s=2\ell+2, 2\ell+3, \ldots, m, $$
where $\alpha_k, \beta_k,\lambda_s$ are real numbers such that
$$\beta_k>0, \quad |\lambda_s|\leq \rho(A), \quad 
\sqrt{\alpha_k^2+\beta_k^2}\leq \rho(A); $$ here $\ell$ is a certain
nonnegative integer. (The existence of such $u_2, \ldots, u_m$ follows
from the canonical form of real normal matrices, see, e.g., \cite{HJ1}.)
Obviously $u_2, \ldots ,u_m$ form an orthonormal basis in
$\mathbb{M}$. Take $$y  = c_1 x + c_2 u_2 + \ldots + c_m u_m\in
\mathcal{K}, $$ thus $c_1\geq\sqrt{c_2^2+\cdots+ c_m^2}$. Then we
have
$$
Ay = \rho(A) c_1 x +
c_2(\alpha_1u_2+\beta_1u_3)+c_3(-\beta_1u_2+\alpha_1u_3) +\cdots
$$ $$ +c_{2\ell}(\alpha_\ell u_{2\ell}+\beta_\ell u_{2\ell+1}) +
c_{2\ell+1}(-\beta_{\ell} u_{2\ell}+\alpha_{\ell}u_{2\ell+1}) + $$
\begin{equation}\lambda_{2\ell+2}c_{2\ell+2}u_{2\ell+2}+ \ldots
+\lambda_m c_mu_m:= \rho(A) c_1 x +w. \label{oct141}\end{equation}
Notice that for $k=1,2,\ldots, \ell$ we have
\begin{eqnarray*}  c_{2k}(\alpha_k u_{2k}+\beta_k u_{2k+1})
&+ & c_{2k+1}(-\beta_{k} u_{2k}+\alpha_{k}u_{2k+1}) \\
&= & (c_{2k}\alpha_k -c_{2k+1}\beta_{k})u_{2k} \ \ +\ \
(c_{2k}\beta_k+c_{2k+1}\alpha_k)u_{2k+1}\end{eqnarray*}
and
$$ \frac{(c_{2k}\alpha_k -c_{2k+1}\beta_{k})^2+
(c_{2k}\beta_k+c_{2k+1}\alpha_k)^2}{\rho(A)^2}=
\frac{(\alpha^2+\beta^2)(c_{2k}^2+c_{2k+2}^2)}{\rho(A)^2}
\leq c_{2k}^2+c_{2k+1}^2. $$
Thus,
\begin{equation}\label{oct142}\|w/\rho(A)\|^2\leq c_2^2+\cdots +c_{2\ell+1}^2+
\frac{\lambda_{2\ell+2}^2}{\rho(A)^2}  c_{2\ell+2}^2  + \cdots+ \frac{\lambda_m^2}{\rho(A)^2} c_m^2 \leq
c_2^2+\cdots
+c_m^2, \end{equation}
and it follows from (\ref{oct141}) and (\ref{oct142}) that $Ay\in \mathcal{K}$.
\bigskip

Assume now that (2) of Theorem \ref{NormCom} holds.
 Let ${\mathcal A}=\{A_1,\ldots, A_q\}$. We may assume
that the spectral radius of each $A_j$ is positive (if
some
$A_j$ is nilpotent, the hypotheses of Theorem \ref{NormCom} (assuming (2)) imply that
it is actually equal to the zero matrix, and
can be ignored). Scaling the $A_j$'s we may further assume that $\rho (A_j)=1$, $j=1,2,\ldots, q$. By Lemma
\ref{2'} we may assume that
$$ A_j=\left[\begin{array}{cc} 1 & 0 \\ 0 & B_j\end{array}\right],
$$
where $B_1,\ldots, B_q$ are $(m-1)\times (m-1) $ matrices. By
Theorem~\ref{th:V}, the hypotheses (1) and (2) of Lemma \ref{2} are
satisfied for $B_1,\ldots, B_q$. Thus, there exists a real positive definite
matrix $V$ such that \eq{6} V-B_j^TVB_j\geq 0, \quad j=1,2,\ldots, q.
\en Then
$$ \KK:=\left\{\left[\begin{array}{c} x \\ y \end{array}\right] \, :\,
x\geq 0, \ \ y\in \mathbb{R}^{m-1} \ {\rm is} \ {\rm such } \ {\rm that} \ y^TVy\leq x^2\right\}
$$
is a common  invariant cone for $A_1, \ldots, A_q$. Indeed, if
$\left[\begin{array}{c} x \\ y \end{array}\right]\in \KK$, then $$A_j
\left[\begin{array}{c} x \\ y
\end{array}\right]=\left[\begin{array}{c} x \\ B_jy
\end{array}\right], $$
and
$$ (B_jy)^TVB_jy\leq
\ {\rm by}\ {\rm (\ref{6})} \ \leq y^TVy\leq x^2, $$ and so $$A_j \left[\begin{array}{c} x \\ y
\end{array}\right]\in \KK. $$
Clearly, $\KK$ is topologically closed, is closed under multiplication
by nonnegative real numbers, is solid and pointed, because of the
positive definiteness of $V$. It remains to prove that $\KK$ is convex.
Thus, let $x_1,x_2\geq 0$ and $y_1,y_2\in \mathbb{R}^{m-1}$ be such
that \eq{7} y_k^TVy_k\leq x_k^2, \quad \mbox{for $k=1,2.$} \en Then
for a number $\alpha$ between $0$ and $1$, we have:
$$ (\alpha y_1+(1-\alpha)y_2)^TV(\alpha y_1+(1-\alpha)y_2)
\leq \alpha^2 x_1^2 +(1-\alpha)^2x_2^2 + 2\alpha (1-\alpha)
(y_1^TVy_2) \leq \ $$  $$ \leq \alpha^2 x_1^2 +(1-\alpha)^2x_2^2 + 2\alpha
(1-\alpha)x_1x_2= (\alpha x_1 +(1-\alpha)x_2)^2
$$ (Cauchy-Schwartz inequality and (\ref{7}) are used in the last step of the derivation),
and the convexity of $\KK$ is proved.
\end{proof}

\section{Examples}\label{examples}
In this section we collect examples that illuminate concepts and
results presented. We use the notation
$$ {\bf e}_1=\left[\begin{array}{c} 1 \\ 0 \end{array}\right], \qquad  {\bf e}_2=\left[\begin{array}{c} 0 \\ 1
\end{array}\right]. $$
\begin{Ex}\label{EMTCounter} Two $2 \times 2$  matrices $A$ and $B$ with negative determinants
such that all words in $A$ and $B$ are Vandergraft matrices  though
there is no $(A,B)$-invariant proper cone.
\end{Ex}

Take \vspace{-.15in}
$$ A=\left[\begin{array}{cc} 1&p \\0&-1 \end{array}\right],\
B=\left[\begin{array}{cc} 1&q \\0&-1\\ \end{array}\right], \quad p\neq
q. $$

All words in $A$ and $B$ are Vandergraft matrices, with $u_1={\bf
e}_1$ as a dominant eigenvector. So, Theorem \ref{th:comeg} applies,
and according to case (iii) in ``Only if'' part of its proof $(A,B)$-invariant
proper cones do not exist. 

Example~\ref{EMTCounter}  shows that Theorem~7.6 in
\cite{EdMcDT} is apparently misstated.

\begin{Ex} \label{PairsNoCommon}
A triple of matrices $T:=\{A, B, C\}, \quad A,B,C \in \pRt$
with the following properties:
\begin{itemize}
\item[${\rm (a)}$] $\det M>0$ for all $M\in T$;
\item[${\rm (b)}$] $A, B, C$ are normal matrices (in particular, diagonalizable);
\item[${\rm (c)}$] there is no $T$-invariant proper cone;
\item[${\rm (d)}$] each pair of matrices in $T$ has a common invariant proper cone;
\item[${\rm (e)}$] no two matrices in $T$ have a common eigenvector.
\end{itemize}
\end{Ex}

The example shows that sharing a common dominant eigenvector is
essential in Corollary~\ref{co:two} and Theorem \ref{NormCom}, and
also that the part of Theorem~\ref{co:4} pertinent to the case when
there are no simultaneously diagonalizable pairs of matrices in
$\mathcal A$ is sharp.

Instead of describing the matrices directly, we will list two linearly
independent eigenvectors and associated eigenvalues for each
matrix. For the eigenvalues simply pick
$\lambda_1(M)>\lambda_2(M)>0$ for each matrix $M \in T$.  As for
the eigenvectors of a matrix $M$, denoting the dominant and non-dominant ones by
$u_1(M)$ and $u_2(M)$ respectively, let
\[ u_1(A)={\bf e}_1,\ u_1(B)=\left[\begin{matrix}1\\ 2\end{matrix}\right],\ u_1(C)=\left[\begin{matrix}1\\
-2\end{matrix}\right] \] and
\[ u_2(A)={\bf e}_2,\ u_2(B)=\left[\begin{matrix}-2\\ 1\end{matrix}\right],\ u_2(C)=\left[\begin{matrix}2\\
1\end{matrix}\right]. \]  Each of the pairs $(A,B)$, $(A,C)$ and $(B,C)$
then satisfies conditions of Corollary~\ref{co:4diagnonneg}, and
therefore has a common invariant proper cone (more specifically, $
{\Cone}\{u_1(A), u_1(B)\}$ is $(A,B)$-invariant, $ {\Cone}\{-u_1(B),
u_1(C)\}$ is \newline $(B,C)$-invariant, and $ {\Cone}\{u_1(A),
u_1(C)\}$ is $(A,C)$-invariant). On the other hand, the separation
condition (ii) of Corollary~\ref{co:4diagnonneg} does not hold for the
triple $(A,B,C)$, so that there is no $(A,B,C)$-invariant proper cone. 

\begin{Ex} \label{4 Needed}
A quadruple of matrices $A,B,C,D\in\R^{2\times 2}$ with distinct
positive eigenvalues such that each triple of them has a common
invariant proper cone while there is no $(A,B,C,D)$-invariant proper
cone. \end{Ex}  In accordance with Theorem~\ref{co:4}, this
quadruple consists of two pairs of commuting matrices.

As in Example~\ref{PairsNoCommon}, the eigenvalues of the matrices
can be chosen arbitrarily, as long as they are positive and distinct.
Following the eigenvector notation from the same Example, let
\[ u_1(A)=u_2(B)={\bf e}_1, \ u_2(A)=u_1(B)={\bf e}_2,\]
\[ u_1(C)=u_2(D)={\bf e}_1+{\bf e}_2, \ u_2(C)=u_1(D)={\bf e}_1-{\bf
e}_2. \] The vectors ${\bf e}_1, {\bf e}_2, {\bf e}_1+{\bf e}_2, {\bf
e}_1-{\bf e}_2$ are simultaneously dominant and non-dominant for
the quadruple $(A,B,C,D)$, and cannot be separated in the sense of
condition (iii) of Theorem~\ref{th:necond}. Consequently, there is no
$(A,B,C,D)$-invariant proper cone. On the other hand, from
Corollary~\ref{co:4diagnonneg} it follows (and can also be checked
directly, based on Theorem~\ref{th:pd}) that $\Cone\{{\bf e}_1, {\bf
e}_2\}$ is $(A,B,C)$-invariant, $\Cone\{{\bf e}_1, -{\bf e}_2\}$ is
$(A,B,D)$-invariant, $\Cone\{{\bf e}_1+{\bf e}_2, {\bf e}_1-{\bf e}_2\}$ is
$(A,C,D)$-invariant, and $\Cone\{{\bf e}_1+{\bf e}_2, {\bf e}_2-{\bf
e}_1\}$ is $(B,C,D)$-invariant.

\begin{Ex} \label{3Violated}
The set $S=\{A,B\}$ which
satisfies all the hypotheses  of Theorem ${\rm \ref{NormCom}}$ (with ${\rm (2)}$ holding) except that
$\rho(A)$, $\rho(B)$ are not semisimple eigenvalues of $A$, $B$, respectively,
and there is no $(A,B)$-invariant proper cone.
\end{Ex}
Take
$$ A=\left[\begin{array}{cc} 1&1 \\0&1 \end{array}\right] \;\;\; , \;\;\;
 B=\left[\begin{array}{cc} 1&-1 \\0&1\\ \end{array}\right]. $$
Clearly, both matrices are Vandergraft, non-diagonalizable, sharing the
dominant eigenline but having different orientation. By
Theorem~\ref{th:necond}, there is no common invariant proper cone.

\begin{Ex}
Two diagonal matrices $A_1$ and $A_2$ without a common invariant
proper cone such that all words in $A_1$ and $A_2$ are Vandergraft
matrices:
\end{Ex}
$$ A_1=\left[\begin{array}{ccc} 1 & 0 & 0 \\ 0 &
-1 & 0\\ 0 & 0& -1 \end{array}\right],\quad A_2=\left[\begin{array}{ccc} -1 & 0 & 0 \\ 0 & -1 & 0\\ 0 & 0 & 1
\end{array}\right]. $$

It is easy to check that all words in $A_1$ and $A_2$ are Vandergraft
matrices. However, condition (\ref{feb91}) of Theorem~\ref{th:diag}
fails, so that  there is no $(A_1,A_2)$-invariant
proper cone.

\begin{Ex}\label{inf} Countable set of $2\times 2$ Vandergraft matrices such that every finite number of them
has a common invariant proper cone, but the whole set does not:
\end{Ex}
Using Theorem \ref{th:pd}, it is easy to see that any  set
of the form $$ \left\{A_m=\left[\begin{array}{cc} 1 & q_m \\ 0 & r \end{array}\right], \qquad m=1,2,\ldots,
\right\}, $$
where the sequence $\{|q_m|\}_{m=1}^\infty$ tends to infinity and $0\leq r<1$ is fixed, fits the bill.

\begin{rem}From standard compactness considerations it follows that if $\mathcal A$ is an infinite family in $\pRn$
any finite subfamily of which has a common invariant proper cone,
then there exists  a non-trivial (that is, different from $\{0\}$) $\mathcal
A$-invariant closed convex pointed cone. However, it may not be
solid, and therefore is not necessarily proper. \end{rem} This is exactly
what is happening in Example~\ref{inf}.



\bibliographystyle{elsarticle-num}

\begin{thebibliography}{10}
\expandafter\ifx\csname url\endcsname\relax
  \def\url#1{\texttt{#1}}\fi
\expandafter\ifx\csname urlprefix\endcsname\relax\def\urlprefix{URL }\fi

\bibitem{BePle94}
A.~Berman, R.~J. Plemmons, Nonnegative matrices in the mathematical sciences,
  SIAM, Philadelphia, PA, 1994, revised reprint of the 1979
  original.

\bibitem{BaRa97}
R.~B. Bapat, T.~E.~S. Raghavan, Nonnegative matrices and applications, Cambridge University
  Press, Cambridge, 1997.

\bibitem{Tam01}
B.-S. Tam, A cone-theoretic approach to the spectral theory of positive linear
  operators: the finite-dimensional case, Taiwanese J. Math. 5~(2) (2001)
  207--277.

\bibitem{Bir67}
G.~Birkhoff, Linear transformations with invariant cones, Amer. Math. Monthly
  74 (1967) 274--276.

\bibitem{Vander}
J.~S. Vandergraft, Spectral properties of matrices which have invariant cones,
  SIAM J. Appl. Math. 16 (1968) 1208--1222.

\bibitem{TaSch}
B.-S. Tam, H.~Schneider, On the core of a cone-preserving map, Trans. Amer.
  Math. Soc. 343~(2) (1994) 479--524.

\bibitem{Tam04}
B.-S. Tam, The {P}erron generalized eigenspace and the spectral cone of a
  cone-preserving map, Linear Algebra Appl. 393 (2004) 375--429.

\bibitem{Bar72}
G.~P. Barker, On matrices having an invariant cone, Czechoslovak Math. J.
  22(97) (1972) 49--68.

\bibitem{VaFa}
M.~E. Valcher, L.~Farina, An algebraic approach to the construction of
  polyhedral invariant cones, SIAM J. Matrix Anal. Appl. 22~(2) (2000) 453--471.

\bibitem{TiFu}
A.~Tiwari, J.~Fung, Polyhedral cone invariance applied to rendezvous of
  multiple agents, in: 43rd IEEE Conference on Decision and Control, Vol.~1,
  Dec. 2004, pp. 165--170.

\bibitem{EdMcDT}
R.~Edwards, J.~J. McDonald, M.~J. Tsatsomeros, On matrices with common
  invariant cones with applications in neural and gene networks, Linear Algebra
  Appl. 398 (2005) 37--67.

\bibitem{BloNe05}
V.~D. Blondel, Y.~Nesterov, Computationally efficient approximations of the
  joint spectral radius, SIAM J. Matrix Anal. Appl. 27~(1) (2005) 256--272.

\bibitem{HJ1}
R.~A. Horn, C.~R. Johnson, Matrix Analysis, Cambridge University Press, New
  York, 1985.

\end{thebibliography}

\end{document}